\documentclass[12pt, twoside]{article}
\usepackage{amsmath,amsthm,amssymb}
\usepackage{times}
\usepackage{enumerate}
\usepackage{extarrows}
\usepackage{color}
\usepackage{tabularx}
\usepackage{multirow}
\usepackage{graphicx}
\DeclareMathOperator*{\esssup}{ess\,sup}

\def\titlerunning#1{\gdef\titrun{#1}}
\makeatletter
\def\author#1{\gdef\autrun{\def\and{\unskip, }#1}\gdef\@author{#1}}
\def\address#1{{\def\and{\\\hspace*{18pt}}\renewcommand{\thefootnote}{}%
\footnote {#1}}%
\markboth{\autrun}{\titrun}}
\makeatother
\def\email#1{e-mail: #1}
\def\subjclass#1{{\renewcommand{\thefootnote}{}%
\footnote{\emph{Mathematics Subject Classification (2010):} #1}}}

\newtheorem{theorem}{Theorem}[section]
\newtheorem{lemma}[theorem]{Lemma}
\newtheorem{definition}[theorem]{Definition}
\newtheorem{proposition}[theorem]{Proposition}
\newtheorem{remark}[theorem]{Remark}

\newtheorem{example}[theorem]{Example}

\newcommand{\R}{\mathbb{R}}

\newcommand{\Proof}{\begin{proof}}
\newcommand{\End}{\end{proof}}

\numberwithin{equation}{section}

\newcommand{\PreserveBackslash}[1]{\let\temp=\\#1\let\\=\temp}
\newcolumntype{C}[1]{>{\PreserveBackslash\centering}p{#1}}
\newcolumntype{R}[1]{>{\PreserveBackslash\raggedleft}p{#1}}
\newcolumntype{L}[1]{>{\PreserveBackslash\raggedright}p{#1}}

\newcolumntype{I}{!{\vrule width 1pt}}
\newlength\savedwidth

\begin{document}


\baselineskip=15pt


\titlerunning{Weak KAM solutions of Hamilton-Jacobi equations}

\title{Nonlinear semigroup approach to Hamilton-Jacobi equations---A toy model}

\author{Liang Jin, Jun Yan \,and\, Kai Zhao}

\date{\today}

\maketitle

\address{Jin Liang: Department of Mathematics, Nanjing University of Science and Technology, Nanjing 210094, China;
 \email{jl@njust.edu.cn }
		\and Jun Yan: School of Mathematical Sciences, Fudan University, Shanghai 200433, China; 
		\email{yanjun@fudan.edu.cn}
		\and
		Kai Zhao:  School of Mathematical Sciences, Fudan University, Shanghai 200433, China;
		\email{zhao$\_$kai@fudan.edu.cn}}

\begin{abstract}
  In this paper, we discuss the existence and multiplicity problem of viscosity solution to the Hamilton-Jacobi equation
  \[
  h(x,d_x u)+\lambda(x)u=c,\quad x\in M,
  \]
  where $M$ is a closed manifold and $\lambda:M\rightarrow\R$ changes signs on $M$, via nonlinear semigroup method. It turns out that a bifurcation phenomenon occurs when parameter $c$ strides over the critical value. As an application of the main result, we analyse the structure of the set of viscosity solutions of an one-dimensional example in detail.
\end{abstract}

%


\section{Introduction}
\setcounter{equation}{0}
\setcounter{footnote}{0}

Let $M$ be a smooth, connected, compact Riemannian manifold without boundary. We use $T^*M$ to denote its cotangent bundle and $H$ a continuous function, called Hamiltonian, on $T^*M\times\mathbb{R}$. The problem of existence and uniqueness of viscosity solution of the Hamilton-Jacobi equation
\begin{equation}\label{HJs}\tag{HJs}
H(x,d_x u,u)=c,\quad x\in M,
\end{equation}
has attracted much attention in past forty years. For fixed constant $c\in\R$, the earliest results are obtained by M.Crandall, P.L.Lions in \cite{CL}-\cite{Crandall_Evans_Lions1984} when $H$ is strictly increasing in $u$, for instance $H(x,p,u)=u+h(x,p)$. The corresponding analytic tools including comparison principle have a great influence on the later development of the viscosity solution theory. For $H=H(x,p)$ independent of $u$, the situation is a bit complicated. A breakthrough was made in \cite{LPV_Hom}, where Lions and his coauthors changed the strategy and successfully proved the solvability of the ergodic problem, i.e., the existence of a pair $(u,c)\in C(M)\times\R$ solving the equation \eqref{HJs}. On the other hand, examples lead to the failure of the uniqueness of solution in this case.

\vspace{1em}
The picture for $u$-independent Hamiltonian becomes more clear after A.Fathi's work in the late 90s. In fact, Fathi built a connection, i.e., weak KAM theory \cite{Fathi_book}, between the theory of viscosity solution and Aubry-Mather theory in Hamiltonian dynamics. It turns out, under suitable assumptions (H1)-(H2) listed below, the constant $c$ found in \cite{LPV_Hom} is uniquely determined by $H$. The ingredients of Fathi's theory consist of regarding the solution of \eqref{HJs} as the large time limit of a nonlinear solution semigroup $\{T^-_t\}_{t\geqslant0}$ generated by the evolutionary equation
\begin{equation}\label{HJe}\tag{HJe}
\begin{cases}
\partial_t u+H(x,\partial_x u,u)=c,\quad (x,t)\in M\times[0,+\infty),\\
u(0,x)=\varphi(x)\in C(M).
\end{cases}
\end{equation}
It is curious to notice that the application of nonlinear semigroup method on the existence problem of evolutionary Hamilton-Jacobi equations already occurred in \cite[VI.3, page 39-41]{CL}. Nevertheless, due to the lack of explicit formula for the semigroup as well as further information on the dynamics of the associated system, the convergence of semigroup was not treated until the birth of weak KAM theory. According to the work of H.Ishii \cite{Ishii_chapter}, most of the weak KAM theory can be fit into the theory of viscosity solution by using  delicate  analytic tools.

\vspace{1em}
More recently, the nonlinear semigroup method was extended to genuinely $u$-dependent Hamiltonian in the sequence of works \cite{WWY1,WWY2,WWY3} by using a new variational principle. It is of particular interest that, based on the works mentioned before, the structure of the set of solutions of \eqref{HJs} can be sketched if $H$ is uniformly Lipschitz in $u$. This includes the untouched case that $H$ is strictly decreasing in $u$. Shortly after \cite{WWY2} occurred, \cite{JMT} generalized the results to ergodic problems from PDE aspects. In this paper, we show, through a simple model, that the results obtained in \cite{WWY1}-\cite{WWY3} allow us to treat the solvability of \eqref{HJs} for any fixed $c\in\R$ when $u$-monotonicity of Hamiltonian is not assumed, and secondly, to present a \textbf{bifurcation phenomenon for the family of equations \eqref{HJs} parametrized by $c$}.

\vspace{1em}
Once and for all, we use $|\cdot|_x$ to denote the dual norm induced by the Riemannian metric on $T^\ast_xM$ and normalize this metric so that diam$(M)=1$. We consider the Hamiltonian $H:T^*M\times\mathbb{R}\to \mathbb{R}$ written in the form
\begin{equation}\label{model}
H(x,p,u)=h(x,p)+\lambda(x)u
\end{equation}
where $h(x,p),\lambda(x)$ are $C^3$ functions satisfying:
\begin{enumerate}
	\item[(H1)] (Convexity) the Hessian $\frac{\partial^2 h}{\partial p^2}$ is positive definite for all $(x,p)\in T^*M$;
	\item[(H2)] (Superlinearity) for every $K\geqslant 0$, there is   $C^{\ast}(K)>0$ such that $h(x,p)\geqslant K|p|_x-C^{\ast}(K)$;
	\item[(H3)] (Fluctuation) there exist $x_1,x_2\in M$ such that $\lambda(x_1)=1$ and $\lambda(x_2)=-1$.
\end{enumerate}
The arguments for establishing our main theorem depend on the variational principal developed in \cite{WWY1}-\cite{WWY3}. This makes our standing assumptions (H1)-(H3) relatively stronger than the standard assumptions in PDE (convexity and coercivity in $p$). With these settings, our results can be summarise into
\begin{theorem}\label{main}
For the equation \eqref{HJs} with Hamiltonian \eqref{model} satisfying (H1)-(H3),
\begin{enumerate}
  \item there is $c(H)\in\R$, uniquely determined by $H$, such that \eqref{HJs} admits a solution if and only if $c\geqslant c(H)$. Furthermore, if $c>c(H)$, \eqref{HJs} admits at least two solutions.
  \item for any $c \geqslant c(H)$, there is a constant $B(H,c)>0$ such that any solution $v:M\rightarrow\R$ to \eqref{HJs} satisfies
      \[
      \|v\|_{W^{1,\infty}(M) }\leqslant B.
      \]
      where $ \|v\|_{W^{1,\infty}(M)} := \esssup_{M}(|v|+|D v| ) $.
\end{enumerate}
\end{theorem}
\vspace{1em}
Here and anywhere, solutions to \eqref{HJs} and \eqref{HJe} should always be understood in the viscosity sense. The remaining of this paper is organized as follows. In Section 2, we briefly recall some necessary tools from \cite{WWY1}-\cite{WWY3} and give a relatively self-contained proof of the main result. Section 3 is devoted to detailed analysis of the structure of the solutions of an example, thus illustrating the meaning of our result.

\section{Proof of the main result}
We divide the proof of main result into two steps. As the first step, we define the constant $c(H)$ and prove its finiteness. When $c<c(H)$, the non-existence of solution of \eqref{HJs} is a direct consequence of that. Secondly, we use tools from the former works \cite{WWY1}-\cite{WWY3} to show the existence and multiplicity of solutions. We use $\|u\|_\infty$ to denote the $C^0$-norm of $u$ as a continuous function on $M$.

\subsection{Critical value and subsolutions to \eqref{HJs}}
In a similar way with \cite{CIP},  we define the critical value of $H$ by
\begin{equation}\label{dcv}
c(H):=\inf_{u\in C^\infty(M)}\sup_{x\in M}H(x,d_x u(x),u(x)).
\end{equation}
Here, we want to remark that for a general Hamiltonian satisfying (H1)-(H2), the number $c(H)$ is not always finite, as the simple example $H(x,p,u)=u+h(x,p)$ shows (in this case, $c(H)=-\infty$). Nevertheless, for the Hamiltonian \eqref{model},

\begin{lemma}\label{finite}
$-\infty <c(H)<+\infty$.
\end{lemma}

\begin{proof}
We choose $u\equiv0$ on $M$, by \eqref{dcv}, to obtain
\[
c(H)\leqslant\sup_{x\in M}H(x,0,0)=\sup_{x\in M}h(x,0)<+\infty.
\]
Note that by taking   $K=0$ in the assumption (H2), there is $e_0:=C^*(0)>0$ such that
\begin{equation}\label{e_0}
\min_{(x,p)\in T^*M}h(x,p)\geqslant -e_0.
\end{equation}
Now the assumption (H3) implies that there exists $x_0\in M$ such that $\lambda (x_0)=0$. Thus for any $u\in C^\infty(M)$,
\begin{align*}
c(H)=&\inf_{u\in C^\infty(M)}\sup_{x\in M}\,\,[h(x,d_x u(x))+\lambda(x)u(x)]\\
\geqslant &\,\inf_{u\in C^\infty(M)}\,\,[h(x_0,d_x u(x_0))+\lambda(x_0)u(x_0)]\\
=&\,\inf_{u\in C^\infty(M)}h(x_0,d_x u(x_0))\geqslant -e_0.
\end{align*}
\end{proof}

An immediate corollary of Lemma \ref{finite} is
\begin{theorem}
For $c<c(H)$, there is no continuous subsolution to the equation \eqref{HJs}.
\end{theorem}
For its proof, we need a standard approximation lemma. We omit the proof of the lemma and refer to \cite[Theorem 8.1]{FM_noncompact} for details.

\begin{lemma}\cite[Lemma 2.2]{DFIZ}
Assume $G\in C(T^{\ast}M)$ such that $G(x,\cdot)$ is convex in $T^{\ast}_x M$ for every $x\in M$, and let $u$ be a Lipschitz subsolution of $G(x,d_x u)=0$. Then, for all $\varepsilon>0$, there exists $u_\varepsilon\in C^{\infty}(M)$ such that $\|u-u_\varepsilon\|_\infty<\varepsilon$ and
$G(x, d_x u_\varepsilon)\leqslant\varepsilon$ for all $x\in M$.
\end{lemma}

\textit{Proof of Theorem 2.2}: From now on, we set
\begin{equation}\label{v-norm}
\lambda_0:=\|\lambda\|_\infty\geqslant1.
\end{equation}
Assume for $c<c(H)$, the equation \eqref{HJs} admits a continuous  subsolution $u:M\rightarrow\R$. Then for any $p\in D^+ u(x)$,
\begin{align*}
h(x,p)\leqslant c-\lambda(x)u(x)\leqslant c-\lambda_0\|u\|_\infty.
\end{align*}
Combining (H2) and the above inequality, we conclude that $u$ is Lipschitz. (A rigorous treatment can be found in \cite[Proposition 1.14]{Ishii_chapter}) Applying Lemma 2.3 to
\[
G(x,p):=h(x,p)+\lambda(x)u(x)-c,
\]
then for $\varepsilon=\frac{1}{2(1+\lambda_0)}(c(H)-c)>0$, there is $u_\varepsilon\in C^\infty(M)$ such that $\|u-u_\varepsilon\|_\infty<\varepsilon$ and
\[
h(x,d_x u_\varepsilon(x))+\lambda(x)u(x)\leqslant c+\varepsilon.
\]
Thus we obtain
\begin{align*}
&H(x,d_x u_\varepsilon(x),u_\varepsilon(x))=h(x,d_x u_\varepsilon(x))+\lambda(x)u_\varepsilon(x)\\
\leqslant &\,h(x,d_x u_\varepsilon(x))+\lambda(x)u(x)+\lambda_0\|u-u_\varepsilon\|_\infty\\
\leqslant &\,c+(1+\lambda_0)\varepsilon<c(H),
\end{align*}
this contradicts \eqref{dcv}.
\qed

\vspace{1em}
The fluctuation condition (H3) gives the existence of subsolutions to \eqref{HJs} when $c$ lies above the critical value. First, we need a priori estimates for subsolutions for \eqref{HJs}.

\begin{lemma}\label{uni-Lip}
 The $C^1$ subsolutions of \eqref{HJs} with $c=c(H)+1$ are equi-bounded and equi-Lipschitzian.
\end{lemma}

\begin{proof}
Let $v\in C^1(M)$ be a subsolution to \eqref{HJs} with $c=c(H)+1$. Due to (H3) and   \eqref{e_0}, we have
\begin{align*}
&-e_0 + v(x_1)\leqslant  h(x_1,d_x v(x_1))+ v(x_1)\leqslant  c(H)+1,\\
&-e_0 - v(x_2)\leqslant  h(x_2,d_x v(x_2))- v(x_2)\leqslant  c(H)+1,
\end{align*}
from which we deduce
\begin{equation}\label{bound}
v(x_1)\leqslant  c(H)+1+e_0,\quad v(x_2)\geqslant-(c(H)+1+e_0).
\end{equation}
Setting $L:=\max_{x\in M}|d_x v(x)|_{x}$, by the mean value theorem and the fact diam$(M)=1$,
\[
|v(x)-v(x_1)|\leqslant L, \quad |v(x_2)-v(x)|\leqslant L,\quad \text{for any}\,\,x\in M.
\]
Combining with \eqref{bound}, this implies that
\begin{equation}\label{eq:1}
\|v\|_\infty\leqslant  |c(H)+1+e_0|+L.
\end{equation}
Thus for $\bar x\in \arg\max_{x\in M} \{|d_xv(x)|_{x}\}$,
\begin{align*}
c(H)+1 \geqslant &\, h(\bar x,d_x v(\bar x))+ \lambda(\bar x)v(\bar x) \\
  \geqslant &\, h(\bar x,d_x v(\bar x))-\lambda_0\|v\|_\infty \\
   \geqslant &\, (1+\lambda_0)L-C^{\ast}(1+\lambda_0)-\lambda_0\|v\|_\infty \\
    \geqslant &\,L-C^{\ast}(1+\lambda_0)-\lambda_0|c(H)+1+e_0|
\end{align*}
where the first inequality follows from \eqref{eq:1} and the second from (H2) with $K=1+\lambda_0$. This implies
\begin{equation}\label{eq:2}
L\leqslant C^{\ast}(1+\lambda_0)+\lambda_0|c(H)+1+e_0|+c(H)+1,
\end{equation}
and the right hand side is independent of $v$. Combining \eqref{eq:1} and \eqref{eq:2} completes the proof .

\end{proof}

\begin{theorem}\label{exist-sub}
There exists a subsolution $v_0\in \mbox{\rm{Lip}}(M)$ to the equation \eqref{HJs} when $c = c(H)$.
\end{theorem}

\begin{proof}
By the definition \eqref{dcv}, for each integer $n\geqslant1$, there exists $v_n\in C^\infty(M)$ such that
\begin{equation}\label{aux-eq1}
\sup_{x\in M}h(x,d_x v_n(x))+ \lambda(x)v_n(x)\leqslant  c(H)+\frac{1}{n}.
\end{equation}
Thus the sequence $\{v_n\}_{n\geqslant1}\subset C^1(M)$ are subsolutions of \eqref{HJs} with $c=c(H)+1$. By Lemma \ref{uni-Lip} and Ascoli-Arzel\`a theorem, it contains a subsequence $\{v_{n_k} \}_{k\in \mathbb{N}}$ uniformly converging on $M$ to some $v_0\in$ Lip$(M)$. Since $v_{n}$ are subsolution of
\[
H(x,d_x u,u)=c(H)+\frac{1}{n},
\]
the stability of subsolutions, see \cite[Theorem 8.1.1]{Fathi_book} or \cite[Theorem 5.2.5]{CS}, implies that $v_0$ is a subsolution of
\[
H(x,d_x u,u)=c(H).
\]
\end{proof}

\subsection{Solution semigroups and their fixed points}
We shall use $TM$ to denote the tangent bundle of $M$. As usual, a point of $T M$ will be denoted by $(x,\dot{x})$, where $x\in M$ and $\dot{x}\in T_x M$. We recall that, for a Hamiltonian $H:T^{\ast}M\times\R\rightarrow\R$ satisfying (H1)-(H2), the corresponding Lagrangian $L:TM\times\R\rightarrow\R$ is defined as
\[
L(x,\dot{x},u)=\sup_{p \in T_x^*M}\{p \cdot \dot{x}-H(x,p,u)\},
\]
i.e., $L$ is the convex dual of $H$ with respect to $p$. Since the equations \eqref{HJs} under consideration are parametrized by $c$, we shall adopt the notions
\[
H^c(x,p,u):=H(x,p,u)-c,\quad L^c(x,\dot{x},u):=L(x,\dot{x},u)+c.
\]
The following action functions are helpful in the definition and estimates of semigroups, they contain important information about the  variational principle defined by equation \eqref{HJe}.
\begin{proposition}\cite[Theorem 2.1, 2.2]{WWY3}\label{Implicit variational}
For any given $x_0\in M$ and $u_0,c\in \R$, there exist continuous functions $h^c_{x_0,u_0}(x,t), h_c^{x_0,u_0}(x,t)$ defined on $M\times (0,+\infty)$ by
\begin{equation}\label{eq:Implicit variational}
\begin{split}
h^c_{x_0,u_0}(x,t)=&\inf_{\substack{\gamma(t)=x\\ \gamma(0)=x_0 } }\Big\{u_0+\int_0^t L^c(\gamma(\tau), \dot \gamma(\tau),h^c_{x_0,u_0}(\gamma(\tau) ,\tau )  )\ d\tau\Big\},\\
h_c^{x_0,u_0}(x,t)=&\sup_{\substack{\gamma(t)=x_0\\ \gamma(0)=x } }\Big\{u_0-\int_0^t L^c(\gamma(\tau), \dot \gamma(\tau),h_c^{x_0,u_0}(\gamma(\tau) ,t-\tau )  )\ d\tau\Big\},
\end{split}
\end{equation}
where the infimum and supremum are taken among Lipschitz continuous curves $\gamma:[0,t]\rightarrow M$ and are achieved. We call $h^c_{x_0,u_0}(x,t)$ the backward action function and $h_c^{x_0,u_0}(x,t)$ the forward action function.
\end{proposition}

\begin{remark}
Let $\gamma\in Lip([0,t],M)$ achieve the infimum (resp. supremum) in \eqref{eq:Implicit variational} and
\[
x(s):=\gamma(s), \quad u(s):=h^c_{x_0,u_0}(x(s),s)\,(\text{resp.}\,\,h_c^{x_0,u_0}(x(s),t-s)), \quad p(s):=\frac{\partial L}{\partial \dot x}(x(s),\dot x(s),u(s)).
\]
Then $(x(s),p(s),u(s))$ satisfies the system
\begin{equation}
 \label{eq:ode}
\left\{
\begin{aligned}
\dot x&=\frac{\partial H }{\partial p}(x,p,u)\\
\dot p &=-\frac{\partial H }{\partial x}(x,p,u)-\frac{\partial H }{\partial u}(x,p,u) \cdot p \quad (x,p,u)\in T^*M \times \R, \\
\dot u&=\frac{\partial H }{\partial p}(x,p,u) \cdot p-H^c(x,p,u)
\end{aligned}
\right.
\end{equation}
with $x(0)=x_0,x(t)=x\,($resp. $x(0)=x,x(t)=x_0)$ and $\lim_{s\to 0^+ }u(s)=u_0\,($resp. $\lim_{s\to t^-}u(s)=u_0)$.
\end{remark}

We collect the properties of the above action functions that are used in this paper into the following
\begin{proposition}{\cite{WWY2}}\label{Minimality}
For each $c\in\R$, the action function $h^c_{x_0,u_0}(x,t)\,\,($resp.\,\,$h^{x_0,u_0}_c(x,t)) $ satisfies
\begin{enumerate}
	\item[(1)] \textbf{(Minimality)} Given $x_0,x\in M $  and $u_0\in \R $ and $t>0$, let $S^{x,t}_{x_0,u_0}$ be the set of the solutions $(x(s),p(s),u(s))$ of \eqref{eq:ode} on $[0,t]$ with $x(0)=x_0,x(t)=x,u(0)=u_0\,\,($resp. $x(0)=x, x(t)=x_0, u(t)=u_0)$. Then
    \begin{equation}\label{eq:Minimality}
    \begin{split}
    h^c_{x_0,u_0}(x,t)=&\inf \{u(t):(x(s),p(s),u(s))\in S^{x,t}_{x_0,u_0}\},\\
    (\text{resp.}\,\,h_c^{x_0,u_0}(x,t)=&\sup \{u(0):(x(s),p(s),u(s))\in S^{x,t}_{x_0,u_0}\}.)
    \end{split}
    \end{equation}
    for any $ (x,t)\in M\times(0,+\infty)$. As a result, $h^c_{x_0,u_0}(x,t)=u\Leftrightarrow h_c^{x,u}(x_0,t)=u_0$.

	\item[(2)]\textbf{(Monotonicity)} Given $x_0\in M, u_1< u_2\in\R$, for any $t>0$ and all $x\in M$,
    \[
	h^c_{x_0,u_1}(x,t)< h^c_{x_0,u_2}(x,t),\quad h_c^{x_0,u_1}(x,t)< h_c^{x_0,u_2}(x,t)
    \]

	\item[(3)]\textbf{(Markov property)} Given $x_0 \in M,u_0 \in \R $, we have
	\begin{equation}\label{markov}
    \begin{split}
	&h^c_{x_0,u_0}(x, t+s)=\inf_{y\in M}h^c_{y,h^c_{x_0,u_0}(y,t)}(x,s),\\
    &h_c^{x_0,u_0}(x, t+s)=\sup_{y\in M}h_c^{y,h_c^{x_0,u_0}(y,t)}(x,s).
    \end{split}
	\end{equation}
	for any $s,t>0$ and all $x\in M$.

	\item[(4)]\textbf{(Lipschitz continuity)} The function $(x_0,u_0,x,t)\mapsto h^c_{x_0,u_0}(x,t)\,\,($resp. $h_c^{x_0,u_0}(x,t))$ is locally Lipschitz continuous on $M\times \R\times M\times (0,+\infty ) $.
\end{enumerate}
\end{proposition}

Based on the backward\,/\,forward action function defined above, we introduce, for each $c\in\R$, two families of nonlinear operators $\{T^{c,\pm}_t\}_{t\geqslant 0}$. For each $\varphi\in C(M)$ and $(x,t)\in M\times(0,+\infty)$,
\begin{equation}\label{eq:Tt-+ rep}
\begin{split}
T^{c,-}_t\varphi(x):=\inf_{y\in M}h^c_{y,\varphi(y)}(x,t),\\
T^{c,+}_t\varphi(x):=\sup_{y\in M}h_c^{y,\varphi(y)}(x,t).
\end{split}
\end{equation}
One easily see that for every $t\geqslant0, T^{c,\pm}_t$ maps $C(M)$ to itself and satisfies for any $t,s\geqslant0$,
\[
T^{c,\pm}_{t+s}=T^{c,\pm}_{t}\circ T^{c,\pm}_{s},
\]
so that the families of operators $\{T^{c,\pm}_t\}_{t\geq0}$ form two semigroups. These semigroups are related to the evolutionary equation \eqref{HJe} by the fact that
\begin{proposition}\label{solution}
Assume that for $\varphi\in C(M)$, $U^{c,\pm}:M\times[0,\infty)\rightarrow\R$ are functions defined by
\[
U^{c,\pm}(x,t):=T^{c,\pm}_t\varphi(x),
\]
then $U^{c,-}$ is the unique solution to \eqref{HJe} and $-U^{c,+}$ is the unique solution to
\begin{equation}\label{reverse ham}
\begin{cases}
\partial_t u+\breve{H}(x,\partial_x u,u)=c,\quad (x,t)\in M\times[0,+\infty),\\
u(0,x)=-\varphi(x),
\end{cases}
\end{equation}
where $\breve{H}(x,p,u)=H(x,-p,-u)$.
\end{proposition}
Due to the above proposition, we call $\{T^{c,-}_t\}_{t\geqslant 0}$ the \textit{backward solution semigroup} and $\{T^{c,+}_t\}_{t\geqslant 0}$ the \textit{forward solution semigroup} to \eqref{HJe}. It turns out that the notion of subsolution (resp. strict subsolution) is equivalent to the $t$-monotonicity (-strict monotonicity) of the solution semigroups. Here, strict subsolutions mean subsolutions that the inequality in the definition of which is strict at any $x\in M$.
\begin{proposition}\label{mono}
Assume $c\geqslant c(H)$, any $v\in C(M)$ is a subsolution to \eqref{HJs} if and only if
\[
v(x)\leqslant\,\,T^{c,-}_t v(x)\quad\text{or}\quad v(x)\geqslant\,\,T^{c,+}_t v(x),\quad\text{for any}\,\,x\in M\,\,\text{and}\,\,t\geqslant0.
\]
Moreover, if $v\in C(M)$ is a strict subsolution, the corresponding \textbf{strictly} inequalities hold.
\end{proposition}
Thus if $T^{c,-}_t v$ \big(resp. $T^{c,+}_t v$\big) has an upper bound (resp. lower bound), then the uniform limits
\[
u^c_-:=\lim_{t\rightarrow\infty}T^{c,-}_t v,\quad u^c_+:=\lim_{t\rightarrow\infty}T^{c,+}_t v
\]
exist and must be fixed points of the corresponding semigroups. This leads to
\begin{definition}\label{weak kam}
We use
\begin{itemize}
  \item  $\mathcal{S}^c_-$ to denote the set of all fixed points of $\{T^{c,-}_t\}_{t\geqslant 0}$, which is also the set of solutions to \eqref{HJs}.
  \item $\mathcal{S}^c_+$ to denote the set of all fixed points of $\{T^{c,+}_t\}_{t\geqslant 0}$. It follows that $u^c_+ \in \mathcal{S}^c_+ $ if and only if $-u^c_+$ is a solution to
     \begin{equation}\label{reverse eq}
      \breve{H}(x,d_x u,u)=c,\quad x\in M;
    \end{equation}
\end{itemize}
\end{definition}

\subsection{Viscosity solutions via solution semigroups}
Since the Hamiltonian has the form \eqref{model}, for $c\geqslant c(H)$,
\[
L^c(x,\dot{x},u)=l^c(x,\dot{x})-\lambda(x)u,
\]
where
\[
l^c(x,\dot{x})=\sup_{p\in T^{\ast}_x M}\{p\cdot\dot{x}-h(x,p)\}+c
\]
is a Tonelli Lagrangian. By Theorem \ref{exist-sub}, there is a subsolution $v_0\in$ Lip $(M)$ to \eqref{HJs} for $c\geqslant c(H)$. The discussion in 2.2 shows that $\lim_{t\rightarrow\infty}T^{c,-}_t v_0$ (if exists!) must be a solution to \eqref{HJs}. By Proposition \ref{mono}, the existence of the limit is equivalent to the the following

\begin{lemma}\label{back-bd}
For $c\geqslant c(H)$ and any subsolution $v$ to \eqref{HJs}, there is $B_0(H,c)>0$ such that
\[
T_{t}^{c,-}v(x)\leqslant  B_0\quad\text{and}\quad T_{t}^{c,+}v(x)\geqslant  B_0 ,\quad \text{for any}\quad x\in M,\,\,t>0.
\]
\end{lemma}

\begin{proof}
We shall focus on the first inequality since the argument for the second one is completely similar.  Since $v$ is a subsolution to \eqref{HJs}, then for any $p\in D^* v(x_1)$, here $D^{\ast}v(x)$ denotes the reachable gradients of $v$ at $x$,
\[
H(x_1,p,v(x_1))=h(x_1,p)+ v(x_1)\leqslant  c,
\]
which is equivalent to
\[
v(x_1) \leqslant  c-\min_{(x,p)\in T^*M}  h(x,p)\leqslant  c+e_0.
\]

By \eqref{eq:Tt-+ rep} and (2) of Proposition \ref{Minimality}, for any $t>0$
\[
T_t^{c,-}v(x)=\inf_{y\in M}h^c_{y,v(y)}(x,t) \leqslant  h^c_{x_1,c+e_0}(x,t).
\]
On the other hand, by \eqref{eq:Implicit variational} and choosing $\gamma_0(\tau)\equiv x_1$ with $\tau\in [0,t]$.
\begin{align*}
h^c_{x_1,c+e_0}(x_1,t)=&\,(c+e_0)+\inf_{\substack{\gamma(t)=x_1\\ \gamma(0)=x_1} }\int_0^t \bigg[l^c(\gamma(\tau),\dot \gamma(\tau))-\lambda(\gamma(\tau))h^c_{x_1,c+e_0}(\gamma(\tau),\tau))\bigg]\ d\tau\\
\leqslant  &\, (c+e_0)+\int_0^t \bigg[l^c(\gamma_0(\tau),\dot \gamma_0(\tau))-\lambda(\gamma_0(\tau))h^c_{x_1,c+e_0}(\gamma_0(\tau),\tau))\bigg]\ d\tau\\
= &\, (c+e_0)+\int_0^t \bigg[l^c(x_1,0) -\lambda(x_1)\cdot h^c_{x_1,c+e_0}(x_1,\tau)\bigg] \ d \tau \\
= &\, (c+e_0)+t\cdot l^c(x_1,0)-\int_0^t h^c_{x_1,c+e_0}(x_1,\tau) \ d \tau.
\end{align*}
We define for $t\geqslant0$,
\[
G(t):=\int_0^t  h^c_{x_1,c+e_0}(x_1,\tau) \ d\tau,\quad G(0)=0.
\]
Then by the above discussion,
\[
\frac{d G(t)}{dt}+G(t) \leqslant  (c+e_0)+t\cdot l^c(x_1,0),
\]
which implies that
\begin{align*}
G(t) \leqslant &\, e^{-t} \int_0^t e^s\Big(c+e_0+s\cdot l^c(x_1,0)\Big)\ ds\\
=&\, \Big(c+e_0-l^c(x_1,0)\Big)\cdot\Big(1-e^{-t}\Big) + t\cdot l^c(x_1,0)\leqslant  D+t\cdot l^c(x_1,0),
\end{align*}
where $D(H):=|e_0-l(x_1,0)|\geqslant 0$. Hence
\[
\int_0^t h^c_{x_1,c+e_0}(x_1,\tau) \ d \tau \leqslant  D+t\cdot l^c(x_1,0).
\]
Therefore, there exists a sequence of positive numbers $\{t_n\}_{n\in\mathbb{N}}$ with $t_n\to +\infty $ and
\[
h^c_{x_1,c+e_0}(x_1,t_n)\leqslant  l^c(x_1,0)+1,
\]
thus by (2) and (3) of Proposition \ref{Minimality}, for any $x\in M$,
\begin{align*}
&\,h^c_{x_1,c+e_0}(x,t_n+1)=\inf_{y\in M}h^c_{y,h^c_{x_1,c+e_0}(y,t_n)}(x,1)\\
\leqslant  &\, h^c_{x_1,h^c_{x_1,c+e_0}(x_1,t_n)}(x,1)
\leqslant  h^c_{x_1,l^c(x_1,0)+1}(x,1)\leqslant  B_0,
\end{align*}
where $B_0:=\max_{x\in M}h^c_{x_1,l^c(x_1,0)+1}(x,1)$ only depends on $H$ and $c$. For any $n\in\mathbb{N}$,
\[
T_{t_n+1}^{c,-}v(x)\leqslant  h^c_{x_1,c+e_0}(x,t_n+1)\leqslant  B_0.
\]
By Proposition \ref{mono}, $T_{t}^{c,-}v(x)$ is increasing in $t$, thus $T_{t}^{c,-}v(x)$ is uniformly bounded from above by $B_0$.
\end{proof}

Combining the above lemma and Proposition \ref{mono}, we obtain
\begin{theorem}\label{exist-sol}
For $c\geqslant c(H)$, both of the limits
\[
u_-^c=\lim_{t\to +\infty} T_t^{c,-}v_0(x), \quad u_+^c=\lim_{t\to +\infty} T_{t}^{c,+}v_0(x)
\]
exists. In particular,\,\,$\mathcal{S}^c_-$ and $\mathcal{S}^c_+$ are both non-empty when $c\geqslant c(H)$.
\end{theorem}

In view of Lemma \ref{uni-Lip}, it is not surprising that $\mathcal{S}^c_-$ is bounded as a subset of Lipschitz functions on $M$. Before the proof of this conclusion, we state the
\begin{lemma}\label{3-3}
For each $t\geqslant  0$ and $x\in M$,
\[
T_t^{c,-}u^c_+(x)\geqslant  u^c_+(x),\quad T_t^{c,+}u^c_-(x)\leqslant u^c_-(x).
\]
\end{lemma}

\begin{proof}
We shall focus one the first inequality, the second one is due to Proposition \ref{mono} and the fact that $u^c_-$ is a subsolution to \eqref{HJs}. It is clear that $T_0^{c,-} u^c_+= u^c_+$. For $t>0$,  we have
\[
T^{c,-}_tu^c_+(x)=\inf_{y\in M}h^c_{y,u^c_+(y)}(x,t),\quad \forall x\in M.
\]
Thus, in order to prove $T_t^{c,-} u^c_+\geqslant   u^c_+$ everywhere, it is sufficient to show that for each $y\in M$,
\begin{equation}\label{eq:7}
h^c_{y, u^c_+(y)}(x,t)\geqslant   u^c_+(x)\quad \text{for all}\quad (x,t)\in M\times (0,+\infty).
\end{equation}
Given any $(x,t)\in M\times (0,+\infty)$, for $y\in M$, set $v(y):=h^c_{y, u^c_+(y)}(x,t)$, then Proposition \ref{Minimality}, (1) gives
\[
u^c_+(y)=h_c^{x,v(y)}(y,t).
\]
Since $u^c_+$ is fixed point of $T_t^{c,+}$, it follows that
\[
h_c^{x,v(y)}(y,t)=u^c_+(y)=T_t^{c,+} u^c_+(y)=\sup_{z\in M}h_c^{z, u^c_+(z)}(y,t)\geqslant  h_c^{x,u^c_+(x)}(y,t).
\]
Proposition \ref{Minimality}, (2) implies
\[
v(y)\geqslant   u^c_+(x)\quad\text{ for all }\quad y\in M,
\]
which is equivalent to \eqref{eq:7}.
\end{proof}

Now we show the second conclusion of our main result, namely
\begin{theorem}\label{Lip-bound}
Assume $c \geqslant  c(H)$, there is $B(H,c)>0$ such that any $u\in \mathcal{S}^c_-$ satisfying
\[
\|u\|_{W^{1,\infty}}\leqslant  B.
\]
\end{theorem}

\begin{proof}
By the superlinearity of $h$, it is enough to show that for any $u\in\mathcal{S}^c_-, \|u\|_{\infty}$ is bounded above by some constant only depending on $H$ and $c$. By Lemma \ref{back-bd}, we obtain that
\begin{equation}\label{eq:thm21-1}
u(x)\leqslant  B_0(H,c), \quad \text{for all}\,\, x\in M,\,\, u\in \mathcal{S}^c_-;\\
\end{equation}
By Lemma \ref{3-3}, any $v\in\mathcal{S}^c_+$ is bounded above by some $u\in \mathcal{S}^c_-$, thus
\begin{equation}\label{eq:thm21-2}
v(x)\leqslant  B_0(H,c), \quad \text{for all}\,\, x\in M,\,\, v\in \mathcal{S}^c_+;\\
\end{equation}
For the other side, we notice that, by Definition \ref{weak kam}, for any $u\in \mathcal{S}^c_-, -u$ is a fixed point of $\{\breve T_t^{c,+}\}_{t\geqslant 0}$ , where $\breve T_t^{c,+}$ is the \textit{forward solution semigroup} to \eqref{HJe}, where $H$ is replaced by $\breve{H}$. Since $\breve{H}$ satisfies our standing assumption (H1)-(H3), we invoke the inequality \eqref{eq:thm21-2} to obtain
\[
-u(x)\leqslant  B_0(\breve{H},c),\quad \text{for any}\quad x\in M.
\]
Combining the above with \eqref{eq:thm21-1}, we conclude that
\[
-B_0(\breve{H},c)\leqslant  u(x)\leqslant  B_0(H,c), \quad \forall x\in M, u\in \mathcal{S}^c_-.
\]
\end{proof}

The following two propositions are crucial in our construction of multiple solutions.
\begin{proposition}\label{principal}
For any $c\geqslant c(H)$ and $u^c_+\in\mathcal{S}^c_+$ (resp. $u^c_-\in\mathcal{S}^c_-$) to the \eqref{HJs}, the limit
\[
\lim_{t\rightarrow\infty}T^{c,-}_t u^c_+\quad\text{(resp.}\,\,\lim_{t\rightarrow\infty}T^{c,+}_t u^c_-)
\]
exists and then belongs to $\mathcal{S}^c_-$ (resp. $\mathcal{S}^c_+$).
\end{proposition}

\begin{proof}
We shall show the first limit exists, the other part is completely similar. Applying Proposition \ref{mono} and Lemma \ref{3-3}, $u^c_+$ is a subsolution to \eqref{HJs}. Then Lemma \ref{back-bd} implies that, for each $t\geqslant 0$,
\[
u^c_+(x)\leqslant  T_t^{c,-}u^c_+(x) \leqslant  B_0(H,c).
\]
Thus $T_t^{c,-}u^c_+(x) $ is increasing and uniformly bounded in $t$ so that $\lim_{t\rightarrow\infty}T^{c,-}_t u^c_+$ exists and has to be a fixed point of $\{T^{c,-}_t\}_{t\geq0}$.
\end{proof}

\vspace{1em}
The following proposition appeared first in the work \cite{WWY3} in a more complete form, which dealt with Hamiltonian that are strictly increasing in $u$. However, its proof does not depend on the $u$-monotonicity. For the readers convenience, we present a simple proof here.
\begin{proposition}\cite[Theorem 1.1, 1.2]{WWY3}\label{mane}
Assume $(u^{c}_-,u^{c}_+)\in\mathcal{S}^c_-\times\mathcal{S}^c_+$ satisfies
\begin{equation}\label{pair}
u^{c}_-=\lim_{t\rightarrow\infty}T^{c,-}_t u^{c}_+,
\end{equation}
then they must coincide on some point on $M$, i.e.,
\[
\mathcal{I}(u_-^{c},u_+^{c}):=\{x\in M:\,\,u_-^{c}(x)=u_+^{c}(x)\}\neq\emptyset.
\]
\end{proposition}

\begin{proof}
We argue by contradiction and assume $ \mathcal{I}(u_-^{c},u_+^{c})=\emptyset$. Then applying Lemma \ref{3-3},
\[
u_-^{c}(x)>u_+^{c}(x)\quad\text{for any}\,\,x\in M,
\]
and $\delta:=\min_{x\in M }\{u_-^{c}(x)-u_+^{c}(x)\}>0$. By the definition \eqref{pair}, there is $t_0>0$ such that for $t\geqslant t_0$,
\begin{equation}\label{strict}
T_t^{c,-} u^{c}_+(x) \geqslant u^{c}_-(x)-\frac{\delta}{2}  > u^{c}_+(x).
\end{equation}
Notice that for any $x\in M,t>0$, the definition \eqref{eq:Tt-+ rep} gives
\[
T_t^{c,-} u^{c}_+(y)\leqslant h^c_{x,u^{c}_+(x)}(y,t).
\]
Thus Proposition \ref{Minimality} implies that
\[
T_t^{c,+} \circ T_t^{c,-} u^{c}_+(x)= \sup_{y\in M} h_c^{y, T_t^{c,-} u^{c}_+(y)} (x,t) \leqslant \sup_{y\in M}  h_c^{y, h^c_{x,u^{c}_+(x)}(y,t)}(x,t)=u^{c}_+ (x).
\]
Combining \eqref{strict} and monotonicity of the solution semigroup, we obtain for $t\geqslant t_0$,
\[
u^{c}_+(x) \geqslant T_t^{c,+} \circ T_t^{c,-} u^{c}_+(x)> T_t^{c,+}   u^{c}_+(x),
\]
This contradicts with the fact that $u^{c}_+ $ is a fixed point of $\{T_t^{c,+}\}_{t\geqslant 0}$.
\end{proof}

As a refined version of the existence result, the multiplicity of solutions to \eqref{HJs} are obtained in the noncritical case. This phenomenon shares similarity with the bifurcation arising in nonlinear dynamics, but has a global nature.
\begin{theorem}\label{multi}
$\mathcal{S}^c_-$ contains at least two elements when $c>c(H)$.
\end{theorem}

\begin{proof}
By the definition \eqref{dcv} of $c(H)$, there is a strict subsolution $v\in C^\infty(M)$ to \eqref{HJs}, i.e.,
\[
H(x,d_x v(x),v(x))<c.
\]
Then by Proposition \ref{mono}, for any $t>0$,
\begin{equation}\label{eq:3}
\begin{split}
T_t^{c,-}v(x)>v(x),\\
T_t^{c,+}v(x)<v(x).
\end{split}
\end{equation}
If we define
\begin{equation}\label{vis-def}
u_-^c:=\lim_{t\to +\infty}T_{t}^{c,-}   v(x)\in \mathcal{S}^c_-, \quad u_+^c:=\lim_{t\to +\infty}T_{t}^{c,+}  v(x)\in \mathcal{S}^c_+.
\end{equation}
and
\begin{equation}\label{ano-def}
\bar{u}_-^c:=\lim_{t\to +\infty}T_{t}^{c,-}u_+^c(x)\in \mathcal{S}^c_-.
\end{equation}
By \eqref{eq:3} and \eqref{vis-def}, for any $x\in M$,
\begin{equation}\label{eq:4}
u_+^c(x)<v(x)<u_-^c(x).
\end{equation}
Due to the definition \eqref{eq:Tt-+ rep} and (2) of Proposition \ref{Minimality}, the above inequality gives for any $t>0$ and $x\in M$,
\[
T_t^{c,-}u_+^c(x)<T_t^{c,-}u_-^c(x)=u_-^c(x),
\]
where the equality holds since elements in $\mathcal{S}^c_-$ are fixed points of   $\{T_t^{c,-}\}_{t\geqslant 0}$. By Proposition \ref{principal}, we send $t$ to infinity and use \eqref{ano-def} to find for any $x\in M$,
\[
\bar{u}_-^c(x)\leqslant  u_-^c(x).
\]
Notice that if $\bar{u}_-^c\equiv u_-^c$ on $M$, then by Proposition \ref{mane}, the set
\[
\mathcal{I}(u_-^c,u_+^c)=\{x\in M:\,\,u_-^c(x)=u_+^c(x)\}\neq\emptyset.
\]
This contradicts with \eqref{eq:4}. Thus $\bar{u}_-^c\leqslant u_-^c$ are different solutions to \eqref{HJs}.
\end{proof}

\section{An illustrating example and concluding remarks}\label{section3}
In this section, we want to illustrate our main theorem by a simple example from another aspect. Let $\mathbb{S}^1$ be the usual circle and $\mathbf{0}\in C(\mathbb{S}^1)$ the function vanishing everywhere on $\mathbb{S}^1$, we consider the following

\begin{example}
\begin{equation}\label{ham}
H(x,p,u)=|p|^2+\sin(x)u,\quad\quad (x,p,u)\in T^\ast\mathbb{S}^1\times\R.
\end{equation}
\end{example}

As a smooth function on $\mathbb{S}^1, \sin(x)$ attains its maximum $1$ and minimum $-1$. It follows that $H$ satisfies (H1)-(H3), thus falls into the class of Hamiltonian we consider in this paper. The Hamilton-Jacobi equation associated to the \eqref{ham} is
\begin{equation}\label{i-e}
|u^\prime(x)|^2+\sin(x)u=c,\quad x\in\mathbb{S}^1.
\end{equation}
From now on, we identify $\mathbb{S}^1$ with $[0,2\pi]$ and all functions are assumed to be $2\pi$-periodic. Notice that due to the facts
\begin{enumerate}[(1)]
  \item if $c=0$, then $\mathbf{0}$ is a solution to \eqref{i-e},
  \item if $c<0$, there is no smooth subsolution $\underline{\mathtt{u}}\in C^\infty(\mathbb{S}^1)$ to \eqref{i-e}. Since if $x=0,\pi$, then $\underline{\mathtt{u}}$ has to satisfy the impossible inequality
        \[
        |\underline{\mathtt{u}}^{\prime}(x)|^2\leqslant c<0,
        \]
\end{enumerate}
and the definition \eqref{dcv}, we have
\begin{proposition}\label{cri}
For the Hamiltonian   $H$ given by \eqref{ham},
\begin{equation}\label{e-cv}
c(H)=0.
\end{equation}
\end{proposition}
In this section, for $c\geq0$, we drop the superscript and use  $u_{\pm}$ and $T_t^{\pm}$ to denote $u^c_{\pm}$ and $T_t^{c,\pm}$ respectively. According to the value of the righthand constant, we divide the discussion into two parts.

\subsection{The critical case $c=0$}
For $c=0$, Theorem \ref{exist-sol} implies that the set $\mathcal{S}^0_-$ of solutions to \eqref{i-e} is non-empty. Since the Hamiltonian $H$ is smooth and strictly convex in $p$, by \cite[Theorem 5.3.7]{CS}, any solution $u_-\in\mathcal{S}^0_-$ is locally semiconcave. In particular,
\begin{equation}\label{d+}
D^+ u_-(x)\neq\emptyset,\quad\text{for any}\,\,x\in[0,2\pi].
\end{equation}

\begin{proposition}\label{nonnegative}
Any $u_-\in\mathcal{S}^0_-$ satisfies
\begin{equation}\label{back}
u_-(x)\equiv0\quad \text{for}\quad x\in[0,\pi],\quad u_-(x)\geqslant0\quad \text{for}\quad x\in[\pi,2\pi].
\end{equation}
\end{proposition}

\begin{proof}
By \eqref{d+}, the condition of being a subsolution implies
\begin{equation}\label{cond-sub}
\sin(x)u_-(x)\leqslant  0,\quad \text{for any}\,\,x\in[0,2\pi].
\end{equation}
From the continuity of $u_-$ and the first inequality above, we have
\begin{equation}\label{eq:5}
\begin{split}
u_-(x)\leqslant0\quad &\text{for}\quad x\in[0,\pi],\quad u_-(x)\geqslant0\quad \text{for}\quad x\in[\pi,2\pi],\\
&\text{so that}\quad u_-(0)=u_-(\pi)=0.
\end{split}
\end{equation}
By \eqref{eq:5}, $u_-$ attains its minimum at $x^{\ast}\in(0,\pi)$. Since a locally semiconcave function is differentiable at a local minima, $u_-$ is differentiable at $x^\ast$ and $u_-^{\prime}(x^\ast)=0$. Since $u_-$ is a supersolution to \eqref{i-e},   one may conclude that $0\in D^-u_-(x^*)$ and
\[
\sin(x^\ast)u_-(x^\ast)=|0|^2+\sin(x^\ast)u_-(x^\ast)\geqslant0,
\]
thus for any $x\in[0,2\pi],\,\,u_-(x)\geqslant u_-(x^\ast)\geqslant0$. Combining this with \eqref{eq:5}, any $u_-\in\mathcal{S}^0_-$ satisfies \eqref{back}.
\end{proof}

\begin{remark}
By repeating the discussion for $H(x,-p,-u)$, it is readily seen that similar conclusion holds true for $u_+\in\mathcal{S}^0_+$, i.e., they all satisfy
\begin{equation}\label{for}
u_+(x)\leqslant 0\quad \text{for}\quad x\in[0,\pi],\quad u_+(x)\equiv0\quad \text{for}\quad x\in[\pi,2\pi].
\end{equation}
\end{remark}

\medskip
Notice that there is a classical solution $\phi\in C^1([\pi,2\pi])$ of
\begin{equation}\label{eq:appendix:ex1}
\phi^\prime (x)= \sqrt{(-\sin x)\phi} \text{ \ in \ } [\pi, 2\pi], \quad \phi(\pi)=0,  \quad\phi(x)>0, \text{ \ if \ } x>\pi
\end{equation}
Indeed, setting
$$
\phi(x)=\Big( \frac{1}{2} \int_\pi^x \sqrt{-\sin t} d t \Big)^2,
$$
we see that $\phi$ is a desired solution of \eqref{eq:appendix:ex1}. We define $u_{\pi,2\pi}:[0,2\pi]\rightarrow\R$ by
$$
u_{\pi,2\pi}(x)=
\begin{cases}
	0, \quad & x\in [0,\pi]\\
	\Big( \frac{1}{2} \int_\pi^x \sqrt{-\sin t} d t \Big)^2, \quad & x\in [\pi,\frac{3}{2}\pi]\\
	\Big( \frac{1}{2} \int_x^{2\pi} \sqrt{-\sin t} d t \Big)^2, \quad &x\in [\frac{3}{2}\pi ,2\pi ],
\end{cases}
$$
then $u_{\pi,2\pi}\in\mathcal{S}^0_-$.
More generally, for $\pi\leqslant a<b \leqslant 2\pi$, we define the function
$$
u_{a,b}(x)=
\begin{cases}
	0, \quad & x\in [0,a]\\
	\min\Big\{  \Big( \frac{1}{2} \int_a^x \sqrt{-\sin t} d t \Big)^2,\Big( \frac{1}{2} \int_x^b \sqrt{-\sin t} d t \Big)^2 \Big\}, \quad & x\in [a,b]\\
	0, \quad &x\in [b,2\pi ]
\end{cases}
$$
and choose a sequence of mutually disjoint interval $(a_i,b_i)\subset(\pi,2\pi)$, with $i\in \mathbb{N}$ or $i\leqslant K$, then the function
$$
\sum_i u_{a_i,b_i}(x)
$$
is a solution to \eqref{i-e}, thus belongs to $\mathcal{S}^0_-$. The following theorem shows all elements of $\mathcal{S}^0_-$ belong to this family. Assume $u_-:[0,2\pi]\rightarrow\R$ is a solution of \eqref{i-e}, by Proposition \ref{nonnegative}, the set
\[
\{x\in [0 ,2\pi],u_-(x)>0\}
\]
is an open subset of $(\pi,2\pi)$, therefore can be written as an at most countable union of mutually disjoint open intervals $(a_i,b_i),i\in I=\mathbb{N}$ or $\{1,2,...,K\}$. We have

\begin{theorem}\label{cri-sol}
Assume $u_-\in\mathcal{S}^0_-$ and
\begin{equation}\label{eq:18}
\{x\in [0 ,2\pi],u_-(x)> 0 \}=\bigcup_{i\in I}(a_i,b_i)\subset(\pi,2\pi),
\end{equation}
then
$$
u_-(x)=\sum_{i} u_{a_i,b_i}(x),\quad x\in [0,2\pi].
$$
\end{theorem}

\begin{proof}
Let $\mathcal{N}_{u_-}=\{x\in [0,2\pi] : u_-(x)=0\}$ denote the null set of $u_-$. Since $\mathcal{N}_{u_-}$ is the set of minima of $u_-$, thus $u_-$ is differentiable at any point $x$ in $\mathcal{N}_{u_-}$ with $u_-^{'}(x)=0$. We use $\mathcal{D}u_-$ to denote the differentiable points of $u_-$. For any $x\in[0,2\pi]$ such that $u'_-(x)$ exists and equals $0$, the equation
\[
0=(u'_-(x))^2+\sin(x)u_-(x)=\sin(x)u_-(x)
\]
and \eqref{back} implies $x\in\mathcal{N}_{u_-}$. Thus we have
\[
\mathcal{N}_{u_-}=\{x\in\mathcal{D}u_- : u'_-(x)=0\}.
\]
It is obvious from the definition \eqref{eq:18} that for any $i\in I$,
\begin{equation}\label{eq:19}
(a_i,b_i)\cap\mathcal{N}_{u_-}=\emptyset.
\end{equation}

\medskip
\textbf{Claim} : For any fixed $i\in I$,  there exists $x_i \in (a_i,b_i)$  such that
$$
u'_-(x)=\begin{cases}
	>0 , \quad x\in (a_i, x_i) \cap \mathcal{D}u_-.\\
	<0 , \quad x\in (  x_i,b_i ) \cap \mathcal{D}u_-.
\end{cases}
$$
\textit{ Proof of the claim:} We argue by contradiction. Assume that there are $\underline{x}<\bar{x}\in(a_i,b_i)$ with
$$
u'_-(\underline{x})<0\quad \quad u'_-(\bar{x})>0,
$$
then $u_{-}|_{[\underline{x},\bar{x}]}$ attains its local minima $x_\ast\in(\underline{x},\bar{x})$. Thus $u_-$ is differentiable at $x_\ast$ with $u_-^{'}(x_\ast)=0$ and $x_\ast\in\mathcal{N}_{u_-}$. This contradicts \eqref{eq:19} and completes the proof.
	
\medskip
	
\medskip
To prove the theorem, it is enough to prove that for each fixed $i\in I$,
\[
u_-|_{[a_i,b_i]}=u_{a_i,b_i}.
\]
Using our claim, we get
\[
\begin{cases}
u_-^\prime (x)=	\sqrt{(-\sin x) u_-(x)}, \quad  &  x\in (a_i, x_i) \cap \mathcal{D}u_-   \\
u_-^\prime (x)= -\sqrt{(-\sin x) u_-(x)}, \quad &   x\in ( x_i,b_i) \cap \mathcal{D}u_-,
\end{cases}
\]
or equivalently
\[
\begin{cases}
(\sqrt{u_-})^\prime (x)=\sqrt{-\sin x}, \quad  &  x\in (a_i, x_i)\cap\mathcal{D}u_-   \\
(\sqrt{u_-})^\prime (x)=-\sqrt{-\sin x}, \quad &   x\in ( x_i,b_i)\cap\mathcal{D}u_-.
\end{cases}
\]
This implies that $\sqrt{u_-}|_{(a_i,x_i)}, \sqrt{u_-}|_{(x_i,b_i)}$ are differentiable almost everywhere with a continuous derivative, thus are absolutely continuous. Applying the fundamental theorem of calculus and the condition $u_-(a_i)=u_-(b_i)=0$, we obtain
\[
u_-(x)=\begin{cases}
\Big( \frac{1}{2} \int_{a_i}^x \sqrt{-\sin t} \ d t \Big)^2, \quad & x\in [a_i,  x_i ]\\
\Big( \frac{1}{2} \int_x^{b_i} \sqrt{-\sin t} \ d t \Big)^2, \quad &x\in [  x_i ,b_i ]
\end{cases}
\]
The above formula and the continuity of $u_-$ at $x_i$ implies that $x_i=\frac{a_i+b_i}{2}$ and
\[
u_-(x)=\min\Big\{  \Big( \frac{1}{2} \int_{a_i}^x \sqrt{-\sin t} d t \Big)^2,\Big( \frac{1}{2} \int_x^{b_i}\sqrt{-\sin t} d t \Big)^2 \Big\}=u_{a_i,b_i}(x),\quad \text{for} \quad x\in [a_i,b_i].
\]
\end{proof}

Theorem \ref{cri-sol} shows that for $c=0$, there are infinite solutions to \eqref{i-e}. However, as a concluding
\begin{remark}\label{unique-prin}
We first observe that for any $u_+\in\mathcal{S}^0_+$ and $u_-\in\mathcal{S}^0_-$,
\[
\lim_{t\rightarrow\infty}T^-_t u_+=\mathbf{0},\quad \lim_{t\rightarrow\infty}T^+_t u_-=\mathbf{0}.
\]
This is due to the following reason: by Definition \ref{principal}, for $x\in[\pi,2\pi]$,
\[
\lim_{t\rightarrow\infty}T^-_t u_+(x) \leqslant \lim_{t\rightarrow\infty}T^-_t \mathbf{0} (x)=0
\]
is a solution to \eqref{i-e}, where the inequality uses \eqref{for}. Applying \eqref{back} to $\lim_{t\rightarrow\infty}T^-_t u_+(x) $ shows that $\lim_{t\rightarrow\infty}T^-_t u_+(x)=\mathbf{0}$. The proof of the second equation is completely similar.

\vspace{1em}
Now if we look for a pair of functions $(u_-,u_+)\in\mathcal{S}^0_-\times\mathcal{S}^0_+$ satisfying
\[
\lim_{t\rightarrow\infty}T^+_t u_-=u_+,\quad \lim_{t\rightarrow\infty}T^-_t u_+=u_-,
\]
then the above observation shows that $(\mathbf{0},\mathbf{0})$ is the unique such pair. In this sense, the results obtained from our main theorem is optimal. The above discussions can be carried out to more general Hamiltonian $H(x,p,u)=|p|^2_x+v(x)u$, where $v:M\rightarrow\R$ is a Morse function and $0$ is a regular value of $v$.
\end{remark}

\subsection{The noncritical case $c>0$}
Assume $c>0$, we want to give a construction of solutions to \eqref{i-e} from the viewpoint of dynamical system and show that $\mathcal{S}^c_-$ contains exactly two elements. For convenience, solutions to \eqref{i-e} are assumed to be $2\pi$-periodic functions on $\R$. For a solution $u_-\in\mathcal{S}^c_-$, we define its graph and 1-graph by
\[
\Lambda^0(u_-)=\{(x,u_-(x))\,:\,x\in\mathbb{R}\}\quad\text{and}\quad\Lambda^1(u_-)=\{(x,p,u)\,:\,u=u_-(x),p\in D^{\ast}u_-(x),x\in\mathbb{R}\}.
\]
To begin with, we notice that a solution $u_-\in \mathcal{S}^c_-$ is a fixed point of backward solution semigroup, i.e.,
\[
T^{-}_t u_-=u_-,\quad\text{for any}\quad t>0.
\]
By Proposition \ref{Implicit variational} and Theorem \ref{Lip-bound}, the above equality implies that for any $x\in\mathbb{R}$, there is an orbit
\[
(x(t),p(t),u(t)),\quad t\in(-\infty,0]
\]
of the characteristic system \eqref{eq:ode} with the contact Hamiltonian
\begin{equation}\label{ham-c}
H^c(x,p,u)=|p|^2+\sin(x)u-c
\end{equation}
such that $x(0)=x,\,p(0)\in D^\ast u_-(x),\,u(t)=u_-(x(t))$, here $D^{\ast}u_-(x)$ denotes the reachable gradients of the solution $u_-$ at $x$. Moreover, for any $t<0$, $x(t)\in\mathcal{D}u_-$ with $u'_-(x(t))=p(t)$.

\begin{remark}
For such an orbit and any $s\leqslant t\leqslant0$, it follows that
\[
u_-(x(t))-u_-(x(s))=\int^{t}_{s} L^c(x(\tau), \dot x(\tau),u_-(x(\tau)))\ d\tau.
\]
Due to this identity, $(x(t),p(t),u(t)),\,t\in(-\infty,0]$ is said to be calibrated by $u_-$ or briefly called calibrated orbit since $u_-$ is fixed.
\end{remark}
Since $u_-$ is a solution to \eqref{HJs}, any calibrated orbit $(x(t),p(t),u(t)), t\in(-\infty,0]$
\begin{enumerate}[(1)]
  \item lies on the regular energy shell $M^c:=\{(x,p,u)\in T^\ast\mathbb{R}\times\R, H(x,p,u)=c\}$, preserved by the contact Hamiltonian flow $\Phi^t_{H^c}$.
  \item and $\alpha(x(0),p(0),u(0))$ is a nonempty connected, $\Phi^t_{H^c}$-invariant subset. Besides, elementary knowledge from topological dynamics shows this set is included in the non-wandering set of $\Phi^t_{H^c}|_{M^c}$.
\end{enumerate}



 When restricting to $M^c$, the system \eqref{eq:ode}, with the Hamiltonian defined by \eqref{ham-c}, has the form
\begin{equation}\label{eq:i-e}
\left\{
\begin{aligned}
\dot x&=2p,\\
\dot p &=-\cos(x)u-\sin(x)p,\\
\dot u&=2p^2.
\end{aligned}
\right.
\end{equation}
We shall use some symmetric property given by the above equation
\begin{lemma}\label{sym}
Assume \eqref{eq:i-e} admits an integral curve $(x(t),p(t),u(t)), t\in I$, where $I$ is an open interval, then
\[
(x(t)+2\pi,p(t),u(t)),\quad (\pi-x(t),-p(t),u(t)),\quad t\in I
\]
are also integral curves of \eqref{eq:i-e}.
\end{lemma}

\begin{proof}
$(x(t)+2\pi,p(t),u(t))$ is an integral curve of \eqref{eq:i-e} since the system \eqref{eq:i-e} is $2\pi$-periodic in $x$. For $t\in I$, we verify by \eqref{eq:i-e} and direct computation as
\[
\left\{
\begin{aligned}
&\frac{d}{dt}(\pi-x(t))=-\dot{x}(t)=-2p(t)=2(-p(t)),\\
&\frac{d}{dt}(-p(t))=-\dot{p}(t)=\cos(x(t))u(t)+\sin(x(t))p(t)\\
&\hspace{1.76cm}=-\cos(\pi-x(t))u(t)-\sin(\pi-x(t))(-p(t)),\\
&\frac{d}{dt}u(t)=2p(t)^2=2(-p(t))^2,
\end{aligned}
\right.
\]
which shows $(\pi-x(t),-p(t),u(t)),t\in I$ is also an integral curve.
\end{proof}

Another simple but cruicial observation from the system \eqref{eq:i-e} is that
\begin{itemize}
  \item along the integral curve, $\dot{u}=2p^2\geqslant0$, so $u(t)$ is nondecreasing.
\end{itemize}
This leads to
\begin{lemma}\label{snw}
The non-wandering set of $\Phi^t_{H^c}|_{M^c}$ consists of hyperbolic fixed points
\[
(\frac{\pi}{2}+2k\pi,0,c),\quad (-\frac{\pi}{2}+2k\pi,0,-c),\quad k\in\mathbb{Z},
\]
whose unstable manifolds are one-dimensional. Moreover, every calibrated orbit lies on the unstable manifold of one of these fixed points.
\end{lemma}

\begin{proof}
Assume an orbit $(x(t),p(t),u(t))$ belongs to the non-wandering set. Then there is $u_0\in\R$ such that $u(t)\equiv u_0$ and $\dot{u}(t)=|p(t)|^2\equiv0$. As a consequence, for some $x_0\in M$,
\begin{equation}\label{eq:6}
\begin{split}
\dot{x}(t)&=2 p(t)\equiv0,\quad x(t)=x_0,\\
\dot{p}(t)&=-\cos(x_0)u_0\equiv0.
\end{split}
\end{equation}
From these, we deduce that $(x_0,0,u_0)$ is a fixed point. According to the second equation of \eqref{eq:6}, we obtain that
\[
x_0=\pm\frac{\pi}{2}+2k\pi\quad\text{or}\quad u_0=0.
\]
Concerning $(x_0,0,u_0)\in M^c$, we have $u_0\neq0$ and the only fixed points are $(\pm\frac{\pi}{2}+2k\pi,0,\pm c)$.

\vspace{1em}
Forgetting the last equation of \eqref{i-e}, we linearize the remaining two on the two dimensional energy shell $M^c$. Set $X=x-x_0, P=p$, we obtain that
\[
\begin{split}
&\left(
  \begin{array}{c}
    \dot{X} \\
    \dot{P} \\
  \end{array}
\right)
=\left(
   \begin{array}{cc}
     0 & 2 \\
     c & -1 \\
   \end{array}
 \right)
\left(
  \begin{array}{c}
    X \\
    P \\
  \end{array}
\right),\quad \text{at}\,\,(x_0,0,u_0)=(\frac{\pi}{2}+2k\pi,0,c);\\
&\left(
  \begin{array}{c}
    \dot{X} \\
    \dot{P} \\
  \end{array}
\right)
=\left(
   \begin{array}{cc}
     0 & 2 \\
     c & 1 \\
   \end{array}
 \right)
\left(
  \begin{array}{c}
    X \\
    P \\
  \end{array}
\right),\quad \text{at}\,\,(x_0,0,u_0)=(-\frac{\pi}{2}+2k\pi,0,-c).
\end{split}
\]
Since $c>0$, it follows that all fixed points are hyperbolic and have one-dimensional stable and unstable manifolds on $M^c$. Since the $\alpha$-limit of every calibrated orbit is connected and included in the non-wandering set of $\Phi^t_{H^c}|_{M^c}$, it has to be one of these fixed points. This completes the proof.
\end{proof}
Now we project the dynamics on $M^c$ onto the $(x,u)$-plane. By the equations $du=pdx$, the projected unstable manifolds of fixed points are 1-graphs of functions near fixed points, but may admit cusp type singularities and turn around in $x$-direction, as is depicted below (contact geometers call these projections \textbf{wave fronts}, and are familiar with their structures, see \cite[Section 1]{ETM}).

\begin{figure}[h]
\begin{center}
\includegraphics[width=9cm]{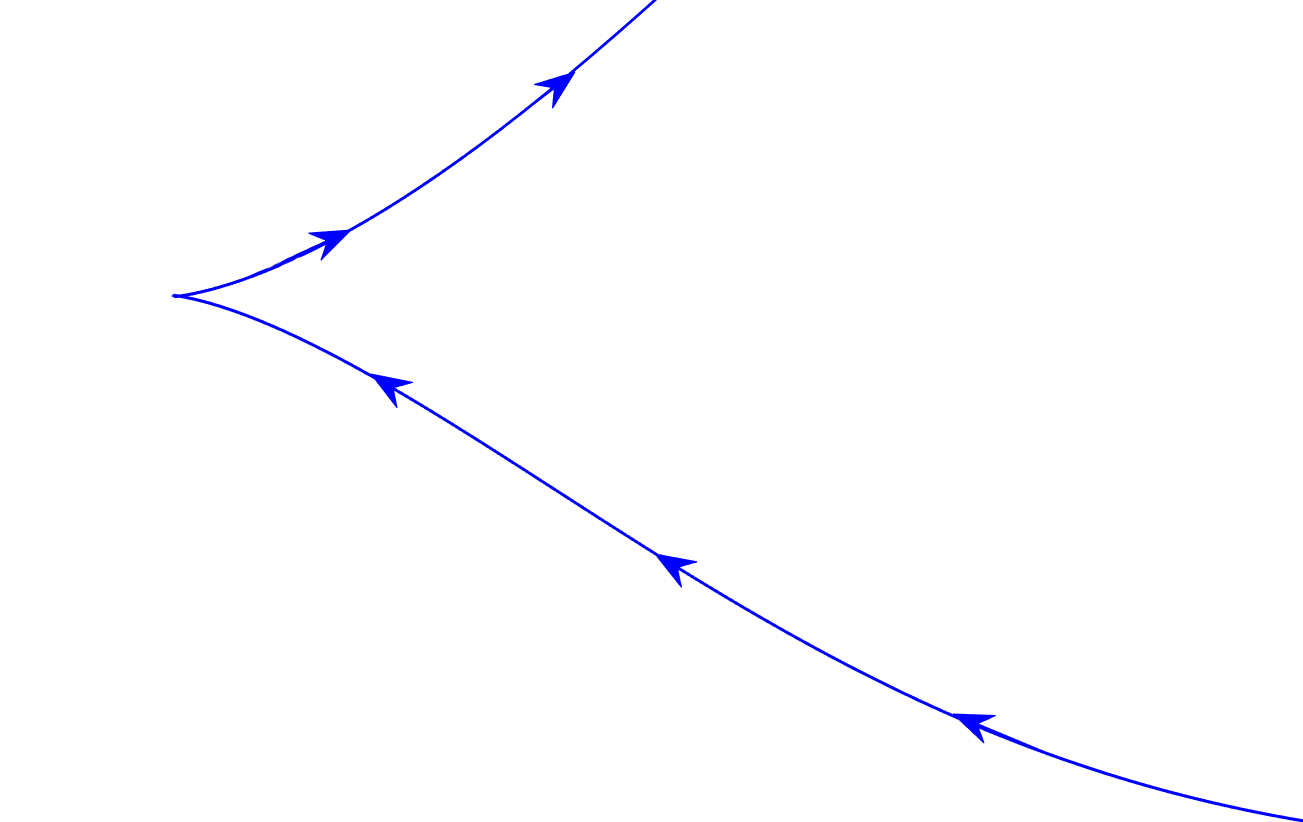}
\end{center}
\end{figure}
Assume $(x(t),p(t),u(t)),t\in(-\infty,0]$ serves as a calibrated orbit of $u_-$, then it lies on the graph of $u_-$ and can not swerve, i.e., $p(t)$ does not change sign on $(-\infty,0]$. Thus we only consider the ``horizontal'' part of the unstable manifolds before turning around and denote them by the notation $\mathrm{W}^u$. According to the equation \eqref{eq:i-e} and Lemma \ref{sym}, we have

\begin{proposition}\label{unstable-prop}
There are $C^1$ functions
\[
\phi_0:\bigg[-\frac{\pi}{2}-\sigma_0,-\frac{\pi}{2}+\sigma_0\bigg]\rightarrow\mathbb{R},\quad \phi_1:\bigg[\frac{\pi}{2}-\sigma_1,\frac{\pi}{2}+\sigma_1\bigg]\rightarrow\mathbb{R}
\]
and $\sigma_0,\sigma_1\in(0,+\infty]$ such that for any $k\in\mathbb{Z}$,
\begin{equation}\label{unstable}
\begin{split}
\mathrm{W}^{u}\bigg(-\frac{\pi}{2}+2k\pi,0,-c\bigg)&=\bigg\{(x,\phi_0(x-2k\pi)):x\in[-\frac{\pi}{2}+2k\pi-\sigma_0,-\frac{\pi}{2}+2k\pi+\sigma_0]\bigg\},\\
\mathrm{W}^{u}\bigg(\frac{\pi}{2}+2k\pi,0,c\bigg)&=\bigg\{(x,\phi_1(x-2k\pi)):x\in[\frac{\pi}{2}+2k\pi-\sigma_1,\frac{\pi}{2}+2k\pi+\sigma_1]\bigg\}
\end{split}
\end{equation}
and
\begin{enumerate}[(1)]
  \item $\phi_0|_{[-\frac{\pi}{2},-\frac{\pi}{2}+\sigma_0]}$ and $\phi_1|_{[\frac{\pi}{2},\frac{\pi}{2}+\sigma_1]}$ are strict increasing and $\phi_0(x)=\phi_0(-\pi-x), \phi_1(x)=\phi_1(\pi-x)$. As a result, $\phi_0$ takes minimum $-c$ only at $x=-\frac{\pi}{2}$ and $\phi_1$ takes minimum $c$ only at $x=\frac{\pi}{2}$.

  \item for any $u_-\in\mathcal{S}^c_-$ and $(x,p,u)\in\Lambda^1(u_-)$, if
        \[
        \alpha(x,p,u)=(-\frac{\pi}{2}+2k\pi,0,-c)\quad\bigg(\text{resp. }(\frac{\pi}{2}+2k\pi,0,c)\bigg),
        \]
        then $\min\{\sigma_0,2\pi\}\geq|-\frac{\pi}{2}+2k\pi-x|\,\,($ resp. $\min\{\sigma_1,2\pi\}\geq|\frac{\pi}{2}+2k\pi-x|\,\,)$ and
        \begin{align*}
        u_-|_{[x,-\frac{\pi}{2}+2k\pi]}\quad\text{or}\quad u_-|_{[-\frac{\pi}{2}+2k\pi,x]}=\phi_0(\cdot-2k\pi),\\
        (\text{resp. }u_-|_{[x,\frac{\pi}{2}+2k\pi]}\quad\text{or}\quad u_-|_{[\frac{\pi}{2}+2k\pi,x]}=\phi_1(\cdot-2k\pi)).
        \end{align*}
\end{enumerate}
\end{proposition}

\begin{proof}
From the form of the equation \eqref{eq:i-e}, we deduce that for any integral curve near the fixed points, $|\frac{du}{dx}(x(t))|=|p(t)|\ll1$. Thus the existence of $C^1$ function $\phi_i$ and constants $\sigma_i$ is a direct consequence of Hartman-Grobman Theorem \cite[Theorem 4.1]{Palis-de Melo} and Lemma \ref{sym}.

\vspace{1em}
(1)\,\,According to the definition of $\phi_i,\sigma_i$, we assume for $x(0)=-\frac{\pi}{2}+\sigma_0$, the integral curve $(x(t),u(t)):=(x(t),\phi_0(x(t))),\,\,t\in(-\infty,0]$ describes half of the projected horizontal unstable manifold of $(-\frac{\pi}{2},0,-c)$. From the proof of Lemma \ref{snw}, it is clear that if $p(t)\geq0$ vanishes on a non-empty open interval if and only if $(x(t),p(t),u(t))$ is a fixed point, this contradicts the assumption. Thus for any $-\infty<t_1<t_2\leq0$,
\[
x(t_2)-x(t_1)=\int^{t_2}_{t_1}p(\tau)d\tau>0,\quad \phi_0(x(t_2))-\phi_0(x(t_1))=u(t_2)-u(t_1)=2\int^{t_2}_{t_1}p^2(\tau)d\tau>0.
\]
This shows that $\phi_0$ is strict increasing on $[-\frac{\pi}{2},-\frac{\pi}{2}+\sigma_0]$. The symmetric property $\phi_0(x)=\phi_0(-\pi-x)$ follows from Lemma \ref{sym}. The corresponding properties for $\phi_1$ is proved by similar arguments.

\vspace{1em}
(2)\,\,Assume for $(x_0,p_0,u_-(x_0))\in\Lambda^1(u_-)$ and some $k\in\mathbb{Z}, \alpha(x_0,p_0,u_-(x_0))=(-\frac{\pi}{2}+2k\pi,0,-c)$. The corresponding calibrated orbit $(x(t),p(t),u(t)), t\in(-\infty,0]$ satisfies $x(0)=x_0$ and for $t\leq0$,
\[
u_-(x(t))=u(t)=\phi_0(x(t)-2k\pi).
\]
We assume $x_0>-\frac{\pi}{2}+2k\pi$, then $x(t)$ is strictly increasing in $t, \lim_{t\rightarrow-\infty}x(t)=-\frac{\pi}{2}+2k\pi$ and
\begin{equation}\label{eq:8}
u_-(x)=\phi_0(x-2k\pi),\quad x\in[-\frac{\pi}{2}+2k\pi,x_0]
\end{equation}
By definition of $\sigma_0$, it is clear that $0<x_0-(-\frac{\pi}{2}+2k\pi)\leq\sigma_0$. Since $\phi_0$ is strictly increasing, if $x_0-(-\frac{\pi}{2}+2k\pi)\geq2\pi$, then $x_0\geq-\frac{\pi}{2}+(2k+2)\pi>-\frac{\pi}{2}+2k\pi$ and by \eqref{eq:8},
\[
u_-(-\frac{\pi}{2}+(2k+2)\pi)=\phi_0(-\frac{\pi}{2}+2\pi)>\phi_0(-\frac{\pi}{2})=u_-(-\frac{\pi}{2}+2k\pi).
\]
This is a contradiction since $u_-$ is $2\pi$-periodic. The conclusions for $\phi_1$ is proved by similar arguments.
\end{proof}

\vspace{1em}
Denoting by $\mathbf{c}$ the constant function taking value $c>0$ everywhere, we notice that $\mathbf{c}$ is a subsolution to \eqref{i-e}. By Proposition \ref{mono} and Lemma \ref{back-bd}, we define
\begin{equation}\label{positive}
u_1:=\lim_{t\to +\infty} T_t^-\mathbf{c}\in\mathcal{S}^c_-\quad\text{and}\quad u_1(x)\geqslant c>0,\quad\text{for all }x\in\mathbb{R},
\end{equation}
thus \eqref{i-e} admits at least one positive function. Furthermore, we have
\begin{lemma}\label{lem:existpositive}
$\sigma_1\geq\pi$ and if we adopt the convention that for $x$ not belonging to the domain of $\phi_i(\cdot-2k\pi), i=0,1$, $\phi_{i}(x-2k\pi)$ is neglected in taking minimum, then
\begin{equation}\label{sol-1}
u_1(x)=\min_{k\in\mathbb{Z}}\{\phi_{1}(x-2k\pi)\}
\end{equation}
is the \textbf{unique} positive solution to \eqref{i-e}.
\end{lemma}

\begin{proof}
Assume $u_-\in\mathcal{S}^c_-$ is positive everywhere. Then the $\alpha$-limit set of any point on $\Lambda^1(u_-)$ is, for some $k\in\mathbb{Z}, \left(\frac{\pi}{2}+2k\pi,0,c \right)$. It follows from Proposition \ref{unstable-prop} and the periodicity of $u_-$ that
\[
\alpha\left(\Lambda^1(u_-)\right)=\bigcup_{k\in\mathbb{Z}}\left(\frac{\pi}{2}+2k\pi,0,c\right)\quad\text{and}\quad \Lambda^0(u_-)\subset\bigcup_{k\in\mathbb{Z}}\mathrm{W}^u \left(\frac{\pi}{2}+2k\pi,0,c\right).
\]
This shows that $\bigcup_{k\in\mathbb{Z}}\mathrm{W}^u \left(\frac{\pi}{2}+2k\pi,0,c\right)$ is connected, therefore $\sigma_1\geqslant\pi$.

\vspace{1em}
Now it suffices to verify \eqref{sol-1} on $[\frac{\pi}{2},\frac{5\pi}{2}]$ with $u_1$ replaced by an arbitrary positive solution $u_-$, since both sides are $2\pi$-periodic functions. By Proposition \ref{unstable-prop} (2), for $x\in[\frac{\pi}{2},\frac{5\pi}{2}]$, there are $p\in D^{\ast}u_-(x), k\in\mathbb{Z}$ such that
\[
\alpha(x,p,u_-(x))=(\frac{\pi}{2}+2k\pi,0,c),\quad k=0,1.
\]
and there is $x_{\ast}\in(\frac{\pi}{2},\frac{5\pi}{2})$ such that
\[
u_-(x)=
\begin{cases}
\phi_1(x),\quad &x\in[\frac{\pi}{2},x_{\ast}],\\
\phi_1(x-2\pi),\quad &x\in[x_{\ast},\frac{5\pi}{2}].
\end{cases}
\]
Thus we obtain
\[
\phi_1(x_\ast)=u_-(x_\ast)=\phi_1(x_\ast-2\pi)=\phi_1(3\pi-x_\ast),
\]
where the last equality uses Proposition \ref{unstable-prop} (1). Since $x_\ast,3\pi-x_\ast\in[\frac{\pi}{2},\frac{\pi}{2}+\sigma_1]$, Proposition \ref{unstable-prop} (1) again (precisely, monotonicity of $\phi_1$) implies $x_\ast=\frac{3\pi}{2}$, and this completes the proof.
\end{proof}

\vspace{1em}
By Theorem \ref{multi}, there are at least two different solutions of equation \eqref{i-e}. Due to Lemma \ref{lem:existpositive}, there exists at least a solution which is not everywhere positive. Lemma \ref{lem:existnotpositive} proves the uniqueness of such solutions. Before that, we need

\begin{lemma}\label{lem:np-structure}
If $u_-\in\mathcal{S}^c_{-}$ is not everywhere positive, then for any $k\in\mathbb{Z}$, $u_-$ is identified with $\phi_0(\cdot-2k\pi)$ near $x=-\frac{\pi}{2}+2k\pi$.
\end{lemma}

\begin{proof}
Assume there is $(x_0,p_0,u_0)\in\Lambda^1(u_-)$ such that $x_0\in[-\frac{\pi}{2}, \frac{3\pi}{2}]$ and $u_0=u_-(x_0)\leq0$. By the equation \eqref{eq:i-e}, we have the $u$-component of $\alpha(x_0,p_0,u_0)$ is not larger than $u_-(x_0)$. Combining Proposition \ref{unstable-prop} (2), $\alpha(x_0,p_0,u_0)=(-\frac{\pi}{2},0,-c)$ or $(\frac{3\pi}{2},0,-c)$. By continuity and periodicity of $u_-, u_-(-\frac{\pi}{2}+2k\pi)=-c$ and there is $0<\delta\leq\min\{\sigma_0,\frac{\pi}{4}\}$ such that $u_-(x)<0$ for $|x+\frac{\pi}{2}-2k\pi|<\delta$. Notice that for different $k, I_k:=(-\frac{\pi}{2}+2k\pi-\delta,-\frac{\pi}{2}+2k\pi+\delta)$ are mutually disjoint and the fixed point with $x$-component falling into $I_k$ is unique, namely $(-\frac{\pi}{2}+2k\pi,0,-c)$. By the fact that the $u$-component of a calibrated orbit $(x(t),p(t),u(t))$, other than fixed points, is strictly increasing, we conclude that if $x(0)\in I_k$, then for $t\leq0$,
\[
x(t)\in I_k\quad\text{and}\quad\alpha(x(0),p(0),u(0))=(-\frac{\pi}{2}+2k\pi,0,-c).
\]
By Proposition \ref{unstable-prop} (2), this implies that $u_-(x)=\phi_0(x-2k\pi)$ for $x\in I_k$.
\end{proof}

\begin{lemma}\label{lem:existnotpositive}
If $u_-\in\mathcal{S}^c_{-}$ is not everywhere positive, then with the same convention as in Lemma \ref{lem:existpositive},
\begin{equation}\label{np:rep}
u_-(x)=u_0(x):=\min_{k\in \mathbb{Z}}\{\phi_0(x-2k \pi), \phi_1(x-2k\pi)\}.
\end{equation}
Notice that since $\sigma_1\geq\pi$, the above minimum is less than $\max_{x\in[-\frac{\pi}{2},\frac{3\pi}{2}]}\phi_1(x)$ and is always attained.
\end{lemma}

\begin{proof}
Since both sides are $2\pi$-periodic functions, it is sufficient to verify \eqref{np:rep} on $[-\frac{\pi}{2},\frac{3\pi}{2}]$. Proposition \ref{unstable-prop} (2) and Lemma \ref{lem:np-structure} shows that: for any $x\in[-\frac{\pi}{2},\frac{3\pi}{2}]$ and $p\in D^\ast u_-(x)$, the $x$-component of $\alpha(x,p,u_-(x))$ must belong to $[-\frac{\pi}{2},\frac{3\pi}{2}]$; combining the fact that $\phi_1-\phi_0$ is strictly decreasing on $[-\frac{\pi}{2},-\frac{\pi}{2}+\sigma_0]$ and $\phi_1-\phi_0(\cdot-2\pi)$ is strictly increasing on $[\frac{3\pi}{2}-\sigma_0,\frac{3\pi}{2}]$, there are $-\frac{\pi}{2}<a\leqslant b<\frac{3\pi}{2}$ such that
\begin{equation}\label{sol-2}
u_-(x)=
\begin{cases}
\phi_0(x),\quad &x\in[-\frac{\pi}{2},a],\\
\phi_1(x),\quad &x\in(a,b),\\
\phi_0(x-2\pi),\quad &x\in[b,\frac{3\pi}{2}].
\end{cases}
\end{equation}
The above formula implies $a\leq-\frac{\pi}{2}+\sigma_0, b\geq\frac{3\pi}{2}-\sigma_0$ and if $a<b$, then
\begin{equation}\label{eq:9}
\phi_0(a)=\phi_1(a), \phi_0(b-2\pi)=\phi_1(b).
\end{equation}
Now there are two cases:

\begin{figure}[h]
\begin{center}
\includegraphics[width=18cm]{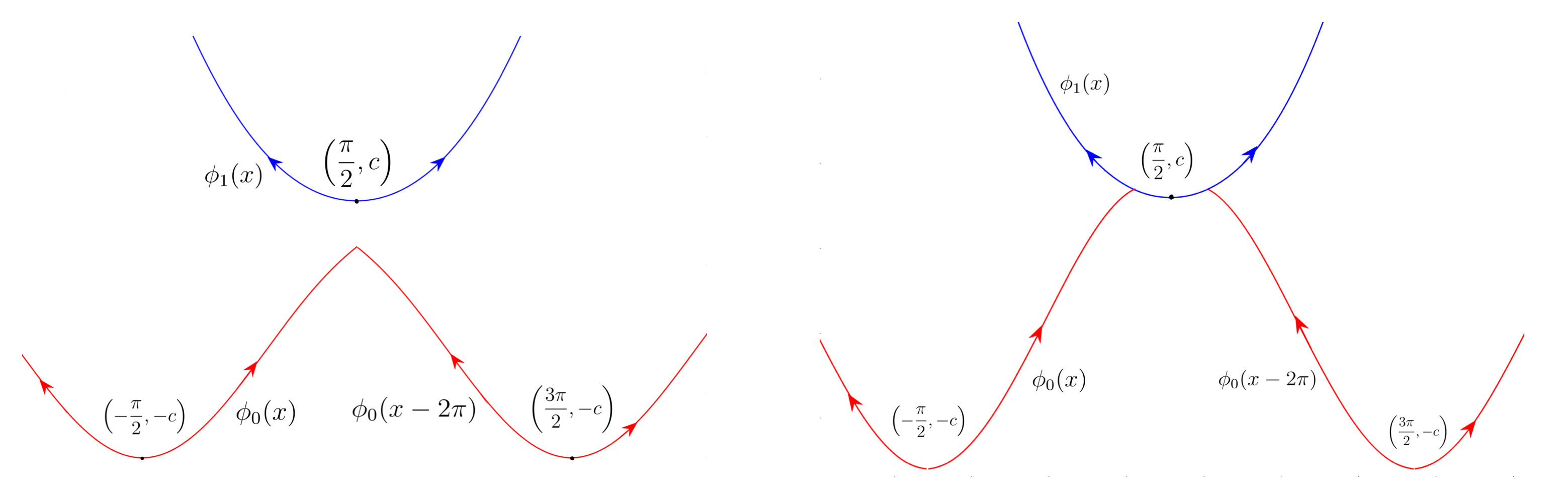}
\caption{\hspace{0.4cm} $a=b$ \hspace{7cm} $a<b$ \hspace{2cm}}
\label{fig1}
\end{center}
\end{figure}

\vspace{1em}
(i)\,\,$a=b$. In this case,
\[
u_-(x)=
\begin{cases}
\phi_0(x),\quad &x\in[-\frac{\pi}{2},a]\\
\phi_0(x-2\pi),\quad &x\in[a,\hspace{0.18cm}\frac{3\pi}{2}]
\end{cases}
\]
and $\frac{3\pi}{2}-\sigma_0\leqslant a\leq-\frac{\pi}{2}+\sigma_0$. So we obtain from continuity of $u_-$ that
\[
\phi_0(a)=u_-(a)=\phi_0(a-2\pi)=\phi_0(\pi-a),
\]
where the last equality uses Proposition \ref{unstable-prop} (1). Since $a,\pi-a\in[-\frac{\pi}{2},-\frac{\pi}{2}+\sigma_0]$, Proposition \ref{unstable-prop} (1) again implies $a=\pi-a=\frac{\pi}{2}, \sigma_0\geq\pi$ and $u_-(x)=\min_{k\in\mathbb{Z}}\{\phi_{0}(x-2k\pi)\}$. Observe that for $p\in D^*u_-( \frac{\pi}{2})$,
\[
\min_{k\in\mathbb{Z}}\{\phi_{0}(x-2k\pi)\}=u_-(x)\leqslant u_-\bigg(\frac{\pi}{2}\bigg)\leqslant p^2+\sin\bigg(\frac{\pi}{2}\bigg)u_-\bigg(\frac{\pi}{2}\bigg)=c\leqslant\min_{k\in\mathbb{Z}}\phi_1(x-2k\pi).
\]
Then $u_-(x)=\min_{k\in\mathbb{Z}}\{\phi_{0}(x-2k\pi)\}=\min_{k\in\mathbb{Z}}\{\phi_{0}(x-2k\pi),\phi_1(x-2k\pi)\}$.

\medskip
(ii)\,\,$a<b$. In this case, $a$ is the \textbf{unique} zero of $\phi_1-\phi_0$ on $[-\frac{\pi}{2},-\frac{\pi}{2}+\sigma_0], b$ is the \textbf{unique} zero of $\phi_1-\phi_0(\cdot-2\pi)$ on $[\frac{3\pi}{2}-\sigma_0,\frac{3\pi}{2}]$. Using Proposition \ref{unstable-prop} (1) and \eqref{eq:9},
\[
0=\phi_1(a)-\phi_0(a)=\phi_1(\pi-a)-\phi_0(-\pi-a)=\phi_1(\pi-a)-\phi_0((\pi-a)-2\pi).
\]
Thus $\pi-a$ is a zero of $\phi_1-\phi_0(\cdot-2\pi)$ on $[\frac{3\pi}{2}-\sigma_0,\frac{3\pi}{2}]$, thus equals $b$. So we obtain
\begin{align}
u_-(x)=
\begin{cases}
\phi_0(x),\quad &x\in[-\frac{\pi}{2},a],\\
\phi_1(x),\quad &x\in(a,\pi-a),\\
\phi_0(x-2\pi),\quad &x\in[\pi-a,\frac{3\pi}{2}].
\end{cases}
\end{align}
By Proposition \ref{unstable-prop} (1), $\phi_0-\phi_1$ is strictly increasing on $x\in [-\frac{\pi}{2},-\frac{\pi}{2}+\sigma_0]$, so for $x\in[-\frac{\pi}{2},\frac{\pi}{2}]$,
\[
u_-(x)=\min_{k\in\mathbb{Z}}\{\phi_{0}(x-2k\pi),\phi_1(x-2k\pi)\}.
\]
The symmetry of $\phi_i$ shows the above identity holds for $x\in[\frac{\pi}{2},\frac{3\pi}{2}]$.
\end{proof}

Finally, we combine Lemma \ref{lem:existpositive} and Lemma \ref{lem:existnotpositive} to obtain
\begin{theorem}\label{thm:existtwosolution}
The equation \eqref{i-e} admit exactly two solutions, i.e.
\[
\mathcal{S}^c_-=\{u_0,u_1 \}
\]
where $u_0=\min_{k\in \mathbb{Z}}\{\phi_0(x-2k\pi),\phi_1(x-2k\pi)\}$ and $u_1=\min_{k\in \mathbb{Z}}\{\phi_1(x-2k\pi)\} $.
\end{theorem}

\begin{remark}
With the aid of numerical method, we depict two solutions $u_0, u_1$ to \eqref{i-e} below in Figure \ref{fig1} with different righthand constants. Notice that when $c=1$, the smooth solution $u_0(x)=\sin(x)$ corresponds to heteroclinic orbits between fixed points. In all cases, the numerical results fit well into our theoretical analysis, especially, Lemma \ref{lem:existpositive} and Lemma \ref{lem:existnotpositive}.
\end{remark}

\begin{figure}[h]
\begin{center}
\includegraphics[width=17cm]{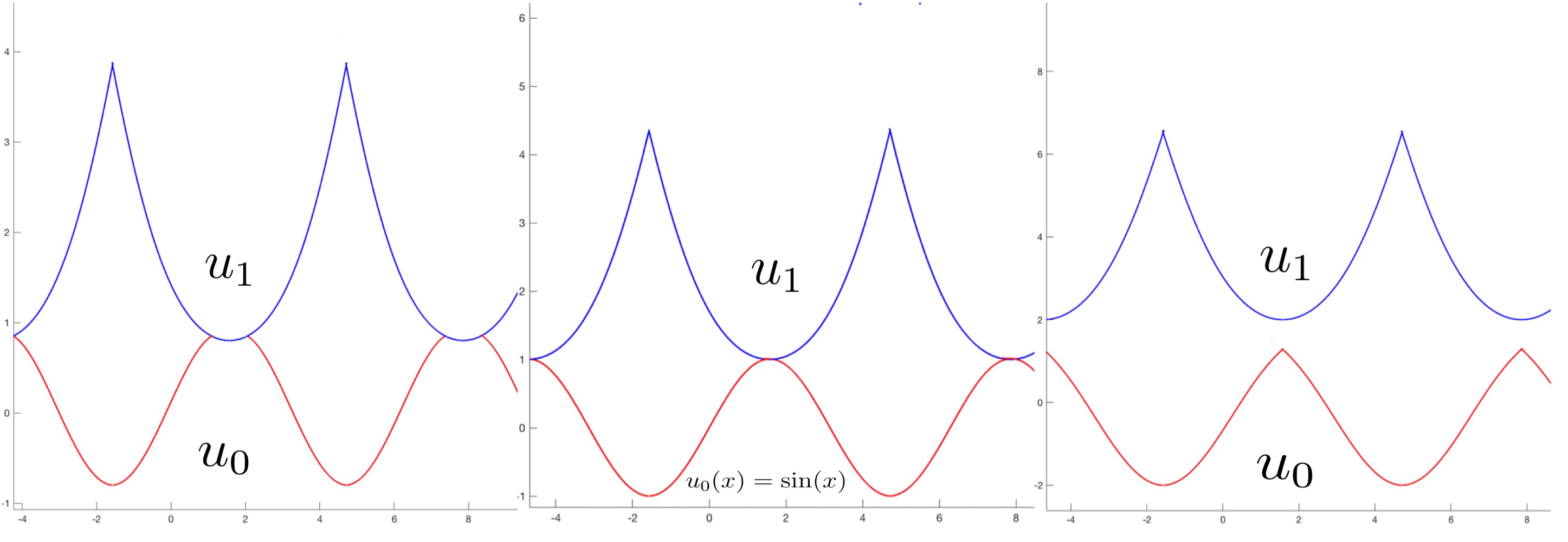}
\caption{ $ \quad c=0.8 \hspace{4cm}  c=1 \hspace{4.5cm}  c=2  \hspace{1.2cm}$  }
\label{fig1}
\end{center}
\end{figure}

\section*{Acknowledgments}
The authors are grateful to the anonymous referees for their careful reading, critical comments and useful suggestions on the original version of this paper, which have helped to improve the presentation substantially. Especially, they kindly pointed out that there are infinite solutions to \eqref{i-e} when $c=0$.  The authors would like to thank Prof. Hitoshi.Ishii warmly for inspiring discussions on the example and numerical results in Section \ref{section3}, which lead them to establish Theorem \ref{cri-sol} and Theorem \ref{thm:existtwosolution}, and for his word by word correction of the manuscript from which they benefit a lot. The authors are partly supported by National Natural Science Foundation of China (Grant No. 12171096). L.Jin is also partly supported by National Natural Science Foundation of China (Grant No. 11901293,11971232).


\end{document}